\documentclass[11pt]{amsart}
\usepackage[T1]{fontenc}
\usepackage{lmodern}
\usepackage{amsmath,amssymb,amsthm,mathtools}
\usepackage{enumitem,microtype,needspace}
\usepackage[margin=1.1in]{geometry}
\usepackage[colorlinks=true,linkcolor=blue,citecolor=blue,urlcolor=blue]{hyperref}

\newtheorem{theorem}{Theorem}[section]
\newtheorem{proposition}[theorem]{Proposition}
\newtheorem{lemma}[theorem]{Lemma}
\newtheorem{corollary}[theorem]{Corollary}
\theoremstyle{definition}

\newtheorem{question}[theorem]{Question}
\theoremstyle{remark}
\newtheorem{remark}[theorem]{Remark}
\numberwithin{equation}{section}

\newcommand{\Z}{\mathbb Z}
\newcommand{\Q}{\mathbb Q}

\newcommand{\cP}{\mathcal P}

\DeclareMathOperator{\Arf}{Arf}

\newcommand{\shn}{\operatorname{sh}}

\newcommand{\gfour}{g_4^{\mathrm{top}}}
\newcommand{\galg}{g_{\mathrm{alg}}}
\newcommand{\gsh}{g_{\mathrm{sh}}}

\title{Topological Shaking Numbers and Cabling}

\author[W.\ Shen]{Weizhe Shen}
\address {Department of Mathematics, University of California, Davis}
\email{wzshen@ucdavis.edu}

\begin{document}

\begin{abstract}
    We give a Seifert-form criterion for upper bounds on topological shaking numbers and prove that isotropic certificates lift under $(m,1)$-cabling.
\end{abstract}

\maketitle

\section{Introduction}\label{sec:introduction}
Let $X_n(K)$ be the $4$-manifold obtained by attaching a $2$-handle to $B^4$ along a knot $K\subset S^3=\partial B^4$ with framing $n\in \Z$ relative to the Seifert framing, and let $C$ be the cocore of the handle. If $K$ is topologically slice, capping a locally flat slice disk for $K$ with the handle core gives a locally flat embedded sphere generating $H_2(X_n(K);\Z)$ and meeting $C$ transversely once. More generally, $K$ is \textit{topologically $n$-shake slice} if $X_n(K)$ contains a locally flat embedded generating sphere, which may meet $C$ more than once. The \textit{topological $n$-shaking number} $\shn_n(K)$ is the minimum of $|S\cap C|$ over all locally flat generating spheres $S\subset X_n(K)$ transverse to $C$, with $\shn_n(K)=\infty$ when no such sphere exists. For each $n\in\Z$, we have $\shn_n(K)=1$ if and only if $K$ is topologically slice.

Feller-Miller-Nagel-Orson-Powell-Ray \cite{FMNOPR} developed a Seifert-surface construction for $1$-shakings. For every knot $K$ with $\Arf(K)=0$, they proved \cite[Proposition 8.8]{FMNOPR} that
\begin{equation}\label{eq:known-bounds}
    2\gfour(K)+1 \le \shn_1(K) \le 2\galg(K)+1,
\end{equation}
where $\gfour$ denotes the topological slice genus and $\galg$ the algebraic genus \cite{FL18,FL19}. The same proposition realizes every positive odd value of $\shn_1$. More generally, for any knot $K$ and integer $h$ with $0<h\le \galg(K)$, their construction gives a locally flat generating surface of genus $h$ in $X_1(K)$ with simply connected complement and exactly $2(\galg(K)-h)+1$ cocore intersections \cite[Corollary 8.5]{FMNOPR}. This also holds for $h=0$ when $\Arf(K)=0$. In \cite[\S8.1]{FMNOPR}, they suggest extending the Seifert-surface approach to framings $|n|\ge 2$ while retaining control of the cocore intersection number.

This paper provides such control for an explicit class of cables. We construct Alexander-trivial subgroups of Seifert forms of shakings and prove a determinant identity that lifts isotropic certificates under $(m,1)$-cabling. Combining this construction with lower bounds, we determine exact shaking numbers and full genus-intersection profiles for explicit families.

All surfaces are connected and oriented, and they are closed unless boundary is specified.
For any knot $K$ and integer $n$, a pair $(g,r)\in\Z_{\ge 0}^2$ is \textit{realized} if a locally flat embedded genus-$g$ surface $S\subset X_n(K)$ represents a generator of $H_2(X_n(K);\Z)$ and meets $C$ transversely in exactly $2r+1$ points. Define the \textit{genus-intersection profile} by
\[
    \cP_n(K)=\{(g,r)\in\Z_{\ge 0}^2 \mid (g,r)\text{ is realized}\}.
\]
Generating surfaces are understood to be locally flat and transverse to the cocore, with no restriction on the fundamental group of their complements. Their geometric intersection number with $C$ is necessarily odd, since their algebraic intersection number is $\pm 1$.

Throughout, we normalize the Alexander polynomial $\Delta_K(t)\in\Z[t,t^{-1}]$ by $\Delta_K(t)=\Delta_K(t^{-1})$ and $\Delta_K(1)=1$. Thus, for any $2g\times2g$ Seifert matrix $V$ for $K$,
\[
    \det(tV-V^T) = t^g\Delta_K(t).
\]
We write $\overline{K}$ for the mirror of a knot $K$.

\begin{theorem}\label{realization}
    Let $n\in\Z\setminus\{0\}$ and $k\in\Z_{\ge 0}$. Define
    \[
        K_{n,k}=
        \begin{cases}
            C_{n,1}(\#^k5_2),&n>0,\\
            \overline{C_{|n|,1}(\#^k5_2)},&n<0.
        \end{cases}
    \]
    Then
    \[
        \begin{gathered}
            \gfour(K_{n,k})=k,\qquad\shn_n(K_{n,k})=2k+1,\\
            \cP_n(K_{n,k})=\{(g,r)\in\Z_{\ge 0}^2 \mid g+r\ge k\}.
        \end{gathered}
    \]
    Moreover, the shaking number is attained by a locally flat generating sphere $S\subset X_n(K_{n,k})$ with $\pi_1(X_n(K_{n,k})\setminus S)\cong \Z/|n|\Z$.
\end{theorem}

Here $C_{m,1}$ denotes the $(m,1)$-cable, with longitudinal winding number $m$, and $\#^0 5_2$ denotes the unknot. The handedness of $5_2$ is fixed in Section \ref{sec:cables}. For $n=1$, the shaking-number formula recovers the calculation for $\#^k5_2$ in \cite[Proposition 8.8]{FMNOPR}. The existence of generating spheres with cyclic complement for these knots already follows from \cite[Corollary 2.1]{FMNOPR} and reflection. Theorem \ref{realization} determines the least cocore intersection number and the full genus-intersection profile of $K_{n,k}$ for every $n\neq 0$.

The upper bound is part of a more general statement.

\begin{theorem}\label{main-cables}
    Let $J$ be a knot with $\Arf(J)=0$ and let $n\in\Z_{\ge 1}$. Then
    \[
        \shn_n(C_{n,1}(J)) \le 2\galg(J)+1 \le 2g_3(J)+1,
    \]
    where $g_3$ denotes the Seifert genus. Moreover, there is a locally flat generating sphere $S\subset X_n(C_{n,1}(J))$ meeting the cocore in exactly $2\galg(J)+1$ points and satisfying $\pi_1(X_n(C_{n,1}(J))\setminus S)\cong \Z/n\Z$. The corresponding statements hold for $\overline{C_{n,1}(J)}$ at framing $-n$.
\end{theorem}

The algebraic input is as follows. Let $V$ be a $2g\times 2g$ Seifert matrix for $K$, and let $n\ge 1$ and $r\ge 0$ be integers. For column vectors $u_1,\dots,u_r,v_1,\dots,v_r\in\Z^{2g}$, set
\begin{equation}\label{eq:intro-E}
    E_n(t) = 
    tV - V^T + (t^n-1)\sum_{l=1}^r\bigl(u_lv_l^T+t^{1-n}v_lu_l^T\bigr).
\end{equation}
Theorem \ref{det-formula} identifies $t^{r(n-1)}\det E_n(t)$ with the Alexander determinant of the Seifert form restricted to an explicit rank-$2(g+r(n-1))$ direct summand for the $(2r+1)$-component $n$-shaking of $K$. If $\det E_n(t)$ is a Laurent monomial with coefficient $\pm1$, then it necessarily equals $t^g$, and we call the vectors a \textit{certificate} for $(V,n)$. The Feller-Lewark construction \cite[Proposition 9 and the proof of Theorem 1]{FL18}, followed by capping with parallel handle-core disks, then gives a locally flat generating sphere in $X_n(K)$ with exactly $2r+1$ cocore intersections and complement fundamental group $\Z/n\Z$.

Theorem \ref{lifting} shows how to lift these certificates under cabling. For integers $m,n'\ge 1$, a certificate for $(V,n')$ satisfying $v_b^T(V-V^T)^{-1}v_a=0$ for all $a,b$ lifts to an isotropic certificate for the $(m,1)$-cable at framing $mn'$, with the same number of vector pairs and determinant $\det E_{n'}(t^m)$. For every knot $J$ with vanishing Arf invariant, we construct an explicit isotropic certificate at framing one for a suitable Seifert matrix of $J$, using $\galg(J)$ pairs. Applying the transfer with $m=n$ and $n'=1$, followed by the geometric criterion, gives Theorem \ref{main-cables}. Signature bounds, together with the genus-intersection inequality, provide the matching lower bounds for Theorem \ref{realization}. The existence of a certificate at framing $n\ge 1$ also forces $\Delta_K(t)\equiv1\pmod{t^n-1}$, as proved in Proposition \ref{certificate-restriction}.

For $|n|\ge 2$, increasing the number of cocore intersections need not allow a reduction in genus. The following theorem gives a signature obstruction and examples at every such framing.

\begin{theorem}\label{main-signature}
    Let $n\in\Z$ with $|n|\ge 2$, and let $K$ be a knot. If $(g,r)\in\cP_n(K)$ and $\xi\neq 1$ is a root of unity of prime-power order satisfying $\xi^{|n|}=1$, then
    \[
        2g\ge |\sigma_K(\xi)|,
    \]
    where $\sigma_K$ denotes the Levine-Tristram signature. In particular, for every $j\in\Z_{\ge 0}$,
    \[
        \cP_n(\#^j5_2)=\{(g,r)\in\Z_{\ge 0}^2 \mid g\ge j\}.
    \]
    If $n$ is even, then for every $j\in\Z_{\ge 0}$ one also has
    \[
        \cP_n(T_{2,2j+1})=\{(g,r)\in\Z_{\ge 0}^2 \mid g\ge j\}.
    \]
\end{theorem}

The lower bound arises because, at the nontrivial $|n|$-th roots of unity of prime-power order, the shaking's additional nullity $2r$ cancels the contribution of its additional boundary components to the first Betti number in the Murasugi-Tristram inequality. Thus the minimum realized genus can remain constant as $r$ increases, even for knots with vanishing Arf invariant. For the knots $K_{n,k}$ in Theorem \ref{realization}, this minimum is instead $\max\{k-r,0\}$, decreasing by one for each unit increase in $r$ until genus zero is reached.

\vspace{6pt}
\paragraph{Organization}
Section \ref{sec:profiles} establishes the general properties of genus-intersection profiles. Section \ref{sec:algebra} proves the determinant formula and the cabling transfer theorem. Section \ref{sec:cables} proves Theorems \ref{realization} and \ref{main-cables}, determines the shaking numbers of the topological-smooth examples from \cite[Corollary 2.3]{FMNOPR}, and gives a restriction for Alexander polynomials of breadth at most two. Section \ref{sec:signatures} proves Theorem \ref{main-signature}. Section \ref{sec:scope} discusses the limitations of the certificates, the zero-framing case, and further questions.

\section*{Acknowledgments}
\noindent The author thanks Patrick Orson and Aru Ray for helpful discussions and for explaining their paper \cite{FMNOPR}.

\section{Genus-intersection profiles}\label{sec:profiles}
We work in the locally flat topological category unless otherwise specified. Write $\gfour(K)$ for the topological slice genus and $g_3(K)$ for the Seifert genus. For $n\in \Z$ and $r\in \Z_{\ge 0}$, define
\[
    \gamma_n(K;r) = \min\{g\in\Z_{\ge 0} \mid (g,r)\in\cP_n(K)\}, 
    \qquad \gsh^n(K) = \min_{r\in\Z_{\ge 0}}\gamma_n(K;r).
\]
Here $2r+1$ is the geometric intersection number with the cocore $C$ of the $2$-handle in $X_n(K)$. The set defining $\gamma_n(K;r)$ is nonempty for every $r$: a capped Seifert surface has one cocore intersection, and local isotopies introduce pairs of opposite-sign intersections without changing the genus. Since the genera are nonnegative integers, both minima are well defined. A knot is $n$-shake slice precisely when $\gsh^n(K)=0$. These definitions impose no restriction on the complement groups of the generating surfaces.

We call $K$ \textit{$\Z/n\Z$-shake slice} if it admits a generating sphere with cyclic complement group. To see why this determines the group, let $\Sigma\subset X_n(K)$ be any generating surface. Since $\Sigma\cdot\Sigma=n$, the relative homology sequence for its exterior identifies the first homology of its complement with the cokernel of multiplication by $n$ on $\Z$. Thus
\[
    H_1\bigl(X_n(K)\setminus\Sigma;\Z\bigr) \cong \Z/n\Z.
\]
A meridian of $\Sigma$ generates this homology group. If the fundamental group of the complement is cyclic, it is therefore isomorphic to $\Z/n\Z$, generated by the same meridian. When $n=0$, this group is $\Z$.

Let $S_{2r+1,n}(K)$ be the link consisting of $2r+1$ parallel copies of $K$ taken with respect to the $n$-framing, with $r+1$ oriented as $K$ and $r$ oppositely. The following form of the shaking correspondence also records the geometric intersection number.

\begin{lemma}[{\cite[Lemma 8.3]{FMNOPR}}]\label{shaking}
    A pair $(g,r)$ lies in $\cP_n(K)$ if and only if $S_{2r+1,n}(K)$ bounds a connected, oriented, locally flat genus-$g$ surface in $B^4$.
\end{lemma}

\begin{proof}
Suppose that a generating surface $\Sigma\subset X_n(K)$ realizes $(g,r)$. Orient $\Sigma$ so that its algebraic intersection with $C$ is $+1$. After straightening $\Sigma$ near its intersections with $C$, choose a sufficiently small product neighborhood of $C$ so that its intersection with $\Sigma$ consists of $2r+1$ transverse disks, one at each intersection point. Removing this neighborhood and rounding corners recovers a $4$-ball and punctures $\Sigma$ along these disks without changing its genus. The resulting boundary consists of $n$-framed parallel copies of $K$. The intersection signs give $r+1$ copies oriented as $K$ and $r$ oppositely, so this boundary is $S_{2r+1,n}(K)$.

Conversely, cap a connected, oriented, locally flat genus-$g$ surface bounded by $S_{2r+1,n}(K)$ with $2r+1$ parallel handle-core disks, using the prescribed orientations. The resulting closed surface is connected and has genus $g$. The algebraic multiplicity of the core disks is $(r+1)-r=1$, so it represents a generator of $H_2(X_n(K);\Z)$. Each core disk meets the cocore once, giving exactly $2r+1$ geometric intersections.
\end{proof}

\begin{proposition}\label{profile-structure}
    For every knot $K$ and every integer $n$, the following hold.
    \begin{enumerate}[label=(\alph*)]
        \item If $(g,r)\in\cP_n(K)$, then $(g+1,r)$ and $(g,r+1)$ lie in $\cP_n(K)$. If $r\ge 1$, then $(g+1,r-1)$ also lies in $\cP_n(K)$.
        
        \item One has $\gamma_n(K;0)=\gfour(K)$ and
        \[
            \gamma_n(K;r) \le \gamma_n(K;r-1) \le \gamma_n(K;r)+1 \qquad(r\ge 1).
        \]
        
        \item Every $(g,r)\in\cP_n(K)$ satisfies $g+r\ge \gfour(K)$. In particular,
        \[
            \shn_n(K)\ge 2\gfour(K)+1.
        \]
        If $\Arf(K)\neq 0$, no generating sphere exists.
        
        \item The profile depends only on the topological concordance class of $K$, and
        \[
            \cP_{-n}(\bar K)=\cP_n(K).
        \]
    \end{enumerate}
\end{proposition}

\begin{proof}
For (a), let $\Sigma$ realize $(g,r)$. Connect-summing $\Sigma$ with a small unknotted torus away from $C$ raises its genus by one without changing its cocore intersections. A local isotopy of $\Sigma$ across $C$ introduces two intersections of opposite signs while keeping $\Sigma$ embedded and preserving its genus.

For the last move, assume $r\ge 1$ and choose intersections $p_+,p_-$ of opposite signs. Such a pair exists because the algebraic intersection is $\pm1$. Join them by an embedded arc $a\subset C$ whose interior avoids $\Sigma$. After straightening $\Sigma$ near the endpoints, restrict a sufficiently small normal $2$-disk bundle of the cocore $C$ to $a$. Its total space is an embedded $3$-ball
\[
    B\cong a\times D^2
\]
meeting $\Sigma$ precisely in the two endpoint disks and meeting $C$ in the zero section $a\times\{0\}$.

Replace the endpoint disks in $\Sigma$ by the lateral annulus $a\times\partial D^2$ and round the corners. The opposite intersection signs ensure that the annulus can be attached compatibly with the orientation of $\Sigma$. The resulting surface is connected, has genus $g+1$, and has exactly two fewer intersections with $C$. The old and new local pieces differ by the oriented boundary of $B$, up to sign, so the homology class is unchanged. This realizes $(g+1,r-1)$.

Since $S_{1,n}(K)$ is isotopic to $K$, the case $r=0$ of Lemma \ref{shaking} identifies $\gamma_n(K;0)$ with $\gfour(K)$. The remaining inequalities in (b) follow by applying the corresponding moves in (a) to genus-minimizing surfaces.

Repeating the last move in (a) exactly $r$ times turns $(g,r)$ into $(g+r,0)$. Part (b) then gives $g+r\ge \gfour(K)$, proving the genus-intersection bound in (c). Applying this bound to a generating sphere gives the asserted lower bound on $\shn_n(K)$, with the convention $\shn_n(K)=\infty$ when no generating sphere exists. The Arf obstruction holds for every integer framing by Robertello's theorem, as proved in \cite[Proposition 6.5]{FMNOPR}, without any assumption on the complement group.

For (d), suppose that $K$ and $K'$ are topologically concordant. A normal push-off extending along the concordance annulus has the same linking number with the respective boundary knot at its two ends. Consequently, the concordance carries the zero framing at one end to the zero framing at the other, and hence carries each integer framing to the same integer framing.

Taking $2r+1$ parallel copies of this annulus, with the prescribed orientations, gives a concordance between $S_{2r+1,n}(K)$ and $S_{2r+1,n}(K')$. Gluing these annuli to a bounding surface preserves its genus and the number of boundary components. Lemma \ref{shaking} therefore gives $\cP_n(K)\subseteq\cP_n(K')$. Applying the reversed concordance gives the reverse inclusion.

Finally, mirroring identifies the orientation-reversed trace $-X_n(K)$ with $X_{-n}(\bar K)$, carrying cocore to cocore and preserving the geometric intersection number. This proves the mirror identity.
\end{proof}

In particular, the profile is determined by the nonincreasing sequence of nonnegative integers $r\mapsto\gamma_n(K;r)$, whose successive drops are at most one. The inequality in part (c) is the genus-intersection lower bound of \cite[Remark 8.6]{FMNOPR}. Its proof above will also be useful in determining exact profiles.

\section{An algebraic criterion and its transfer under cabling}\label{sec:algebra}
We give a Seifert-form criterion for constructing planar surfaces bounded by shakings, and show how isotropic certificates lift under cabling.
We work with positive framings throughout this section. Let $V$ be a $2g\times2g$ integral Seifert matrix, and write $\omega=V-V^T$. Thus $\omega$ is a unimodular alternating matrix and $\det\omega=1$. Fix integers $n\ge 1$ and $k\ge 0$. For indices in $\{0,\ldots,n-1\}$, set
\[
    (Z_n)_{ij}=\delta_{j,i+1\bmod n},\qquad
    M_n=
    \begin{pmatrix}
        0&I_n\\ 
        Z_n&0
    \end{pmatrix},\qquad
    \theta=V\oplus M_n^{\oplus k}.
\]
Within each added summand, this convention for $Z_n$ gives $\theta(\beta_i,\alpha_j)=\delta_{j,i+1\bmod n}$. The matrix $\theta$ is a Seifert matrix of the shaking $S_{2k+1,n}(K)$ whenever $V$ is a Seifert matrix of $K$ \cite[Lemma 8.9]{FMNOPR}. Simultaneously reversing the cyclic indexing of the $\alpha$- and $\beta$-bases identifies our convention with the transpose convention for $Z_n$. The corresponding connected Seifert surface has $2k+1$ boundary components and genus $g+k(n-1)$.

Our geometric aim is to obtain a planar surface bounded by the shaking while retaining its $2k+1$ boundary components. The intersection form $\theta-\theta^T$ of this surface has a radical of rank $2k$, generated by its boundary classes. We seek a complementary direct summand $A$ of rank $2(g+k(n-1))$ on which the Seifert form is Alexander-trivial. An Alexander-trivial summand of this rank can be realized by a subsurface containing all the genus, whose single boundary component has Alexander polynomial one \cite[Proposition 9]{FL18}. Replacing this subsurface by a locally flat disk in $B^4$ produces the desired planar surface. Capping its boundary components with parallel handle-core disks then gives a generating sphere with $2k+1$ cocore intersections. The vectors $u_l,v_l$ introduced below parametrize a family of complements to the radical, and the determinant formula tests when the resulting restriction is Alexander-trivial.

\subsection{A complement to the radical}
Let
\[
    H=\Z^{2g}\oplus D_1\oplus\cdots\oplus D_k, \qquad D_l=\langle\alpha_0^l,\ldots,\alpha_{n-1}^l, \beta_0^l,\ldots,\beta_{n-1}^l\rangle.
\]
We regard $\theta$ as a bilinear form on $H$. Put
\[
    \alpha_\Sigma^l=\sum_{i=0}^{n-1}\alpha_i^l, \qquad \beta_\Sigma^l=\sum_{i=0}^{n-1}\beta_i^l, \qquad R=\langle\alpha_\Sigma^l,\beta_\Sigma^l \mid 1\le l\le k\rangle.
\]

\begin{lemma}\label{radical}
    The radical of $\theta-\theta^T$ is $R$. It is a direct summand of $H$ of rank $2k$. The restriction of $\theta$ to each summand $\langle\alpha_\Sigma^l,\beta_\Sigma^l\rangle$ has matrix $\begin{psmallmatrix}0&n\\ n&0\end{psmallmatrix}$, and distinct such summands are orthogonal.
\end{lemma}

\begin{proof}
The two off-diagonal blocks of $M_n-M_n^T$ are $I_n-Z_n^T$ and $Z_n-I_n$. Each has integral kernel generated by the all-ones vector. Since $\omega$ is nonsingular, this identifies the radical exactly as $R$. Moreover,
\[
    \alpha_0^l,\ldots,\alpha_{n-2}^l,\alpha_\Sigma^l, \quad \beta_0^l,\ldots,\beta_{n-2}^l,\beta_\Sigma^l
\]
is an integral basis of $D_l$, proving that $R$ is a direct summand. The remaining assertions follow by summing the entries of $I_n$ and $Z_n$.
\end{proof}

Choose column vectors $u_l,v_l\in\Z^{2g}$, for $1\le l\le k$. For each standard basis vector $x$ of $\Z^{2g}$, define
\[
    \widehat{x}=
    x + \sum_{l=1}^k \bigl((v_l^Tx)\alpha_\Sigma^l+(u_l^Tx)\beta_\Sigma^l\bigr).
\]
Let $A\le H$ be generated by these vectors, followed, for each $l$, by
\[
    \alpha_0^l,\ldots,\alpha_{n-2}^l, \beta_0^l,\ldots,\beta_{n-2}^l.
\]
Lists with upper index $n-2=-1$ are understood to be empty. Denote the matrix of $\theta|_A$ in this ordered generating set by $\theta_A$, and define the Laurent-polynomial matrix
\begin{equation}\label{eq:certificate-matrix}
    E_{V,n,\boldsymbol u,\boldsymbol v}(t)
    = tV - V^T + (t^n-1)\sum_{l=1}^k \bigl(u_lv_l^T+t^{1-n}v_lu_l^T\bigr).
\end{equation}
We write $E(t)$ when the data are clear. A subgroup is Alexander-trivial if the restriction matrix $B$ of the Seifert form satisfies $\det(tB-B^T)=\pm t^d$ for some integer $d$.

\begin{theorem}\label{det-formula}
    For arbitrary integral vectors $u_l,v_l$, the displayed generators of $A$ form a basis, $H=A\oplus R$, and
    \begin{equation}\label{eq:det-formula}
        \det(t\theta_A-\theta_A^T)
        =t^{k(n-1)}\det E(t).
    \end{equation}
    In particular, $A$ is Alexander-trivial if and only if $\det E(t)$ is a unit in $\Z[t,t^{-1}]$.
\end{theorem}

\begin{proof}
The basis in the proof of Lemma \ref{radical}, together with the standard basis of $\Z^{2g}$, becomes the proposed basis of $A\oplus R$ by adding integral combinations of the generators of $R$ to the standard basis vectors. This is a unimodular basis change, so $H=A\oplus R$.

If $n=1$, there are no remaining $\alpha_i^l,\beta_i^l$ in the basis of $A$, and
\[
    \theta_A=V+\sum_l(u_lv_l^T+v_lu_l^T).
\]
The asserted identity follows immediately. Assume henceforth that $n\ge 2$, and put $p=n-1$. Let $N$ be the $p\times p$ matrix with $N_{ij}=\delta_{i,j+1}$, and let $\mathbf1\in\Z^p$ be the all-ones vector.

The pairings involving the sums of the $\alpha$- and $\beta$-curves are
\[
    \begin{aligned}
        \theta(\alpha_\Sigma,\beta_j)
        &=\theta(\beta_\Sigma,\alpha_j)
        =\theta(\alpha_i,\beta_\Sigma)
        =\theta(\beta_i,\alpha_\Sigma)=1,\\
        \theta(\alpha_\Sigma,\beta_\Sigma)
        &=\theta(\beta_\Sigma,\alpha_\Sigma)=n.
    \end{aligned}
\]
Consequently,
\[
    t\theta_A-\theta_A^T=\begin{pmatrix}X&C\\ C^T&D\end{pmatrix},
\]
where
\[
    \begin{aligned}
        X&=tV-V^T+n(t-1)\sum_l(u_lv_l^T+v_lu_l^T),\\
        C&=(t-1)
        [\,u_1\mathbf1^T\mid v_1\mathbf1^T\mid\cdots
        \mid u_k\mathbf1^T\mid v_k\mathbf1^T\,],\\
        D&=\bigoplus_{l=1}^kD_0,
        \qquad
        D_0=\begin{pmatrix}0&tI_p-N\\ tN^T-I_p&0\end{pmatrix}.
    \end{aligned}
\]
Since $N$ is nilpotent,
\[
    \det D_0 =(-1)^p\det(tI_p-N)\det(tN^T-I_p) =t^p.
\]
Thus $D$ is invertible over $\Q(t)$, and the Schur-complement formula gives
\[
    \det(t\theta_A-\theta_A^T)=t^{kp}\det(X-CD^{-1}C^T).
\]
The individual inverse blocks are
\[
    D_0^{-1}=
    \begin{pmatrix}
        0&(tN^T-I_p)^{-1}\\
        (tI_p-N)^{-1}&0
    \end{pmatrix}.
\]
Writing
\[
    b_p(t)=\sum_{r=0}^{p-1}(p-r)t^r,
\]
and using $\mathbf1^TN^r\mathbf1=p-r$, we obtain
\[
    \begin{aligned}
        \mathbf1^T(tN^T-I_p)^{-1}\mathbf1&=-b_p(t),\\
        \mathbf1^T(tI_p-N)^{-1}\mathbf1&=t^{-1}b_p(t^{-1}).
    \end{aligned}
\]
It follows that
\[
    CD^{-1}C^T=(t-1)^2\sum_l \bigl(-b_p(t)u_lv_l^T+t^{-1}b_p(t^{-1})v_lu_l^T\bigr).
\]
The elementary identities
\[
    \begin{aligned}
        (t-1)^2b_p(t)&=t^{p+1}-(p+1)t+p,\\
        (t-1)^2t^{-1}b_p(t^{-1})&=t^{-p}-(p+1)+pt
    \end{aligned}
\]
show that the respective coefficients of $u_lv_l^T$ and $v_lu_l^T$ in $X-CD^{-1}C^T$ are
\[
    n(t-1)+(t-1)^2b_p(t)=t^n-1
\]
and
\[
    n(t-1)-(t-1)^2t^{-1}b_p(t^{-1}) =t^{1-n}(t^n-1).
\]
Hence $X-CD^{-1}C^T=E(t)$, which proves \eqref{eq:det-formula}.
\end{proof}

We call the vectors $\{u_l,v_l\}_{l=1}^k$ a certificate for $(V,n)$ if $\det E(t)$ is a Laurent monomial with coefficient $\pm1$. The certificate is isotropic if $v_b^T\omega^{-1}v_a=0$ for every pair $a,b$. A certificate necessarily satisfies
\[
    \det E(t)=t^g.
\]
Indeed, if $d=g+k(n-1)$ and $D_A(t)=\det(t\theta_A-\theta_A^T)$, then
\[
    D_A(t)=t^{2d}D_A(t^{-1}).
\]
By Theorem \ref{det-formula}, $D_A(t)$ is a monomial, and this reciprocity identity forces its exponent to be $d$. The same theorem then gives $\det E(t)=\pm t^g$. Finally, $E(1)=\omega$ and $\det\omega=1$, so the coefficient is $+1$.

\subsection{From certificates to embedded spheres}
\begin{corollary}\label{geometric-certificate}
    Suppose $V$ is a Seifert matrix for $K$ and admits a certificate for $(V,n)$ with $k$ pairs of vectors. Then the shaking $S_{2k+1,n}(K)$ bounds a connected, locally flat planar surface $F\subset B^4$ with $\pi_1(B^4\setminus F)\cong\Z$. Capping its boundary components gives a locally flat sphere $S\subset X_n(K)$ representing a generator of $H_2(X_n(K))$, meeting the cocore in $2k+1$ points, and satisfying
    \[
        \pi_1(X_n(K)\setminus S)\cong \Z/n\Z.
    \]
    In particular, the $n$-shaking number of $K$ is at most $2k+1$.
\end{corollary}

\begin{proof}
Let $F_0$ be the connected Seifert surface for $S_{2k+1,n}(K)$ used above, and put $d=g+k(n-1)$. This surface has genus $d$. By Theorem \ref{det-formula}, its Seifert form has an Alexander-trivial direct summand $A$ of rank $2d$. By \cite[Proposition 9]{FL18}, $A$ is realized by a genus-$d$ subsurface $F_1\subset F_0$ with one boundary component and connected complement in $F_0$. The knot $\partial F_1$ has Alexander polynomial one. The construction in the proof of \cite[Theorem 1]{FL18} replaces $F_1$ by a locally flat disk and produces a connected planar surface $F\subset B^4$ with the same oriented boundary and with $\pi_1(B^4\setminus F)\cong\Z$. This is the form of the construction used for shakings in \cite[Section 8]{FMNOPR}.

Cap the $2k+1$ components with parallel core disks in the $2$-handle $D^2\times D^2$. Their orientations consist of $k+1$ positive and $k$ negative copies, so the resulting sphere represents a generator and meets the cocore in $2k+1$ points. We give the fundamental-group calculation explicitly. Let $b=2k+1$ and let $P$ be a disk with $b$ disjoint open subdisks removed. The part of the sphere exterior in the $2$-handle is $D^2\times P$, and its intersection with the exterior of $F$ is $S^1\times P$. Write
\[
    \pi_1(B^4\setminus F)=\langle\mu\rangle, \qquad \pi_1(P)=\langle a_1,\ldots,a_b\rangle.
\]
If $\varepsilon_i\in\{1,-1\}$ is the orientation of the $i$th core disk relative to the product orientation, then $\sum_{i=1}^b\varepsilon_i=1$. Choose $\mu$ to be a positively oriented meridian of $F$. The map from $S^1\times P$ to the exterior of $F$ sends $a_i$ to $\mu^{\varepsilon_i}$. The $S^1$-factor based near the outer boundary of $P$, denoted $\lambda$, follows the $n$-framed longitude of $K$ and has total linking number $n\sum_{i=1}^b\varepsilon_i=n$ with $\partial F$, so it maps to $\mu^n$. On the other side, in $D^2\times P$, $\lambda$ is null-homotopic and the $a_i$ freely generate the fundamental group. The Seifert-van Kampen theorem consequently gives
\[
    \pi_1(X_n(K)\setminus S) =\langle\mu,a_1,\ldots,a_b\mid a_i=\mu^{\varepsilon_i},\ \mu^n=1\rangle \cong\Z/n\Z.
\]
\end{proof}

Certificates are also compatible with connected sum. If $V$ and $V'$ admit certificates with $k$ and $k'$ pairs at the same framing, extend the vectors by zero in $V\oplus V'$. The resulting matrix $E$ is the direct sum of the original two matrices, so it is a certificate with $k+k'$ pairs for $V\oplus V'$. If both original certificates are isotropic, the combined certificate is isotropic as well, since the inverse intersection form is block diagonal.

\subsection{Lifting certificates through cabling}
For $m\ge 1$, define the $m\times m$ block matrix
\[
    \operatorname{cab}_m(V)_{ij}=
    \begin{cases}
        V,&i\le j,\\
        V^T,&i>j.
    \end{cases}
\]
This is a Seifert matrix of the $(m,1)$-cable of $K$, as in the satellite-surface construction in the proof of \cite[Proposition 2.4]{FMP22}. To see the block entries directly, order $m$ parallel copies $F_1,\ldots,F_m$ of a Seifert surface for $K$ in the positive normal direction and join them by the $m-1$ bands of the standard cable surface. If $x_a^{(i)}$ denotes the copy on $F_i$ of the $a$th basis curve, then
\[
    \operatorname{lk}\bigl(x_a^{(i)},(x_b^{(j)})^+\bigr)
    =
    \begin{cases}
        V_{ab},&i\le j,\\
        V_{ba},&i>j.
    \end{cases}
\]
For $i>j$, a negative push-off of the second curve can be exchanged for a positive push-off of the first curve. Symmetry of linking then gives $V_{ba}$. The connecting bands are disjoint from the basis curves, join distinct connected components, and introduce no additional first-homology generators. Thus the resulting surface has genus $mg$ and the displayed matrix has size $2mg$. Its antisymmetrization is the direct sum of $m$ copies of $\omega$.

\begin{theorem}\label{lifting}
    Let $u_l,v_l\in\Z^{2g}$, for $1\le l\le k$, satisfy
    \begin{equation}\label{eq:lifting-isotropy}
        v_b^T\omega^{-1}v_a=0\qquad(1\le a,b\le k).
    \end{equation}
    For $m,n'\ge 1$, set
    \[
        \widetilde u_l=(u_l,\ldots,u_l)^T, \qquad \widetilde v_l=(v_l,0,\ldots,0)^T,
    \]
    where each vector is a column with $m$ blocks. Let $E(t)$ be the matrix in \eqref{eq:certificate-matrix} for $V$ at framing $n'$ with vectors $u_l,v_l$. Let $\widetilde E(t)$ be its analogue for $\operatorname{cab}_m(V)$ at framing $mn'$ with vectors $\widetilde u_l,\widetilde v_l$. Then
    \begin{equation}\label{eq:lifting-identity}
        \det\widetilde E(t)=\det E(t^m).
    \end{equation}
    Thus an isotropic certificate for $(V,n')$ lifts to an isotropic certificate for $(\operatorname{cab}_m(V),mn')$ with the same number of pairs.
\end{theorem}

\begin{proof}
All intermediate matrix calculations take place over $\Q(t)$. Put
\[
    s=t^m,\qquad q=[m]_t=1+t+\cdots+t^{m-1},\qquad B=sV-V^T,\qquad W=t\operatorname{cab}_m(V)-\operatorname{cab}_m(V)^T.
\]
In particular, $B=\omega+(s-1)V$ and $s-1=(t-1)q$.

For a block column $x=(x_1,\ldots,x_m)^T$, write
\[
    S_i=\sum_{j<i}x_j,\qquad T=\sum_{j=1}^m x_j=S_{m+1}.
\]
The $i$th block row of $Wx$ is
\begin{equation}\label{eq:cable-row}
    (Wx)_i=\omega(S_{i+1}-tS_i)+(t-1)VT.
\end{equation}
Multiplying the $i$th row by $t^{m-i}$ and summing gives the telescoping identity
\begin{equation}\label{eq:cable-sum}
    \sum_{i=1}^m t^{m-i}(Wx)_i=BT.
\end{equation}

First, we determine $\det W$. The change of variables from $(x_1,\ldots,x_m)$ to $(S_2,\ldots,S_{m+1})$ has determinant one. Replace the last block row by the weighted sum in \eqref{eq:cable-sum}. Its coefficient on the former last row is one, so this operation also preserves the determinant. The new last row has only the final block $B$. The submatrix obtained by deleting its last block row and column is lower block bidiagonal, with diagonal blocks $\omega$ and subdiagonal blocks $-t\omega$. Therefore
\begin{equation}\label{eq:cable-determinant}
    \det W=(\det\omega)^{m-1}\det B=\det B.
\end{equation}
Because $B|_{t=1}=\omega$, both $B$ and $W$ are invertible over $\Q(t)$.

We next calculate the pairings needed for the matrix determinant lemma. If $Wx=y$, equations \eqref{eq:cable-row} and \eqref{eq:cable-sum} give
\[
    T=B^{-1}\sum_{i=1}^m t^{m-i}y_i, \qquad x_1=\omega^{-1}\bigl(y_1-(t-1)VT\bigr).
\]
For $y=(u,\ldots,u)^T$, these become
\[
    T=qB^{-1}u,\qquad x_1=B^{-1}u.
\]
For $y=(v,0,\ldots,0)^T$, they become
\[
    T=t^{m-1}B^{-1}v, \qquad x_1=\omega^{-1}v- \frac{t^{m-1}}q(\omega^{-1}-B^{-1})v,
\]
where we used $\omega^{-1}VB^{-1}=(\omega^{-1}-B^{-1})/(s-1)$. Consequently,
\begin{align*}
    \widetilde v_b^TW^{-1}\widetilde u_a
        &=v_b^TB^{-1}u_a,\\
    \widetilde u_b^TW^{-1}\widetilde u_a
        &=q\,u_b^TB^{-1}u_a,\\
    \widetilde u_b^TW^{-1}\widetilde v_a
        &=t^{m-1}u_b^TB^{-1}v_a,\\
    \widetilde v_b^TW^{-1}\widetilde v_a
        &=\left(1-\frac{t^{m-1}}q\right)v_b^T\omega^{-1}v_a
          +\frac{t^{m-1}}qv_b^TB^{-1}v_a
          =\frac{t^{m-1}}qv_b^TB^{-1}v_a.
\end{align*}
The last equality is the only use of \eqref{eq:lifting-isotropy} in the determinant calculation.

Set
\[
    c_1=t^{mn'}-1=s^{n'}-1,\qquad c_2=(t^{mn'}-1)t^{1-mn'},\qquad c_2'=(s^{n'}-1)s^{1-n'}=t^{m-1}c_2.
\]
With
\[
    \begin{aligned}
        P&=[c_1\widetilde u_1\mid\cdots\mid c_1\widetilde u_k
        \mid c_2\widetilde v_1\mid\cdots\mid c_2\widetilde v_k],\\
        Q&=[\widetilde v_1\mid\cdots\mid\widetilde v_k
        \mid\widetilde u_1\mid\cdots\mid\widetilde u_k],
    \end{aligned}
\]
we have $\widetilde E=W+PQ^T$. Let $U=[u_1\mid\cdots\mid u_k]$ and $T_v=[v_1\mid\cdots\mid v_k]$, and define
\[
    G=
    \begin{pmatrix}
        c_1T_v^TB^{-1}U&c_2'T_v^TB^{-1}T_v\\
        c_1U^TB^{-1}U&c_2'U^TB^{-1}T_v
    \end{pmatrix},
    \qquad
    \Lambda=\begin{pmatrix}I_k&0\\0&qI_k\end{pmatrix}.
\]
The four pairing identities above say exactly that $Q^TW^{-1}P=\Lambda G\Lambda^{-1}$. On the other hand,
\[
    E(s)=B+c_1UT_v^T+c_2'T_vU^T,
\]
and $G$ is its matrix in the determinant lemma. Applying that lemma twice and using \eqref{eq:cable-determinant}, we obtain
\[
    \begin{aligned}
        \det\widetilde E(t)
        &=\det W\det(I_{2k}+\Lambda G\Lambda^{-1})\\
        &=\det B\det(I_{2k}+G)
        =\det E(s)
        =\det E(t^m).
    \end{aligned}
\]
For $k=0$, the same argument reduces to \eqref{eq:cable-determinant}, with empty matrices in the determinant lemma.

Finally, since the antisymmetrization of $\operatorname{cab}_m(V)$ is the direct sum of $m$ copies of $\omega$ and each $\widetilde v_l$ is supported in its first block, we have
\[
    \widetilde v_b^T\bigl(\operatorname{cab}_m(V)-\operatorname{cab}_m(V)^T\bigr)^{-1}\widetilde v_a =v_b^T\omega^{-1}v_a=0.
\]
Thus the lifted data remain isotropic.
\end{proof}

\section{Cables and exact shaking numbers}\label{sec:cables}

We first construct the isotropic data needed to apply Theorem \ref{lifting}. The construction is an explicit matrix version of the vanishing-Arf argument underlying \cite[Corollary 8.5]{FMNOPR}.

\begin{lemma}\label{arf-certificate}
    Let $J$ be a knot with $\Arf(J)=0$, and put $h=\galg(J)$. There is a Seifert matrix $V$ of size $2G$, with $G=d+h$ for some $d\ge 0$, in an integral basis
    \[
        x_1,\dots,x_{2d},a_1,\dots,a_h,b_1,\dots,b_h,
    \]
    such that its restriction $P$ to the $x$-coordinates is Alexander-trivial, the $x$-subgroup is orthogonal to the $(a,b)$-subgroup for $\omega=V-V^T$, the $(a,b)$-basis is symplectic with $\omega(a_i,b_j)=\delta_{ij}$, and every $V(a_i,a_i)$ is even. For a column vector $u$, write $u(y)$ for its coordinate at the basis element $y$. Set $v_i=e_{a_i}$ and define $u_i$ by
    \[
        u_i(x_j)=-V(a_i,x_j),\qquad
        u_i(b_j)=-V(a_i,b_j),\qquad
        u_i(a_j)=
        \begin{cases}
        -V(a_i,a_j),&j>i,\\
        -V(a_i,a_i)/2,&j=i,\\
        0,&j<i.
        \end{cases}
    \]
    Then
    \[
        v_i^T\omega^{-1}v_j=0\quad\text{for all }i,j, \qquad \det\left(tV-V^T+(t-1)\sum_{i=1}^h (u_iv_i^T+v_iu_i^T)\right)=t^G.
    \]
    Thus these vectors give an isotropic certificate at framing $1$ with $h$ pairs.
\end{lemma}

\begin{proof}
By the definition of algebraic genus \cite{FL18}, there is a genus-$G$ Seifert surface $F$ for $J$ and an Alexander-trivial subgroup $L\subset H_1(F;\Z)$ of rank $2d$, with $G-d=h$. The restriction of the intersection form $\omega$ to $L$ is unimodular: its determinant is obtained by evaluating the Alexander determinant at $t=1$. Consequently there is an integral orthogonal decomposition
\[
    H_1(F;\Z)=L\oplus L^{\perp_\omega}.
\]
Indeed, for any homology class, unimodularity gives a unique element of $L$ with the same intersection pairings against $L$. Subtracting that element gives its component in $L^{\perp_\omega}$. The complementary summand has rank $2h$ and unimodular intersection form.

The quadratic refinement
\[
    q\colon H_1(F;\mathbb F_2)\longrightarrow\mathbb F_2, \qquad q(x)=V(x,x)\bmod 2,
\]
splits as an orthogonal sum over these two summands. The restriction to $L$ has Arf invariant zero because its Seifert form $P$ has normalized Alexander polynomial one. Since $\Arf(J)=0$, additivity of the Arf invariant shows that the restriction to $L^{\perp_\omega}$ also has Arf invariant zero. 
Choose a symplectic basis of this latter quadratic space over $\mathbb F_2$ whose first $h$ elements have $q$-value zero. We lift this basis to an integral symplectic basis of $M=L^{\perp_\omega}$ by induction on $h$. 
For $h>0$, choose a primitive integral lift $a$ of the first basis vector and any integral lift $b_0$ of its symplectic partner. A primitive lift exists by taking coordinates in $\{0,1\}$ in an integral basis of $M$. Since $\omega$ is unimodular, there is a vector $c\in M$ with $\omega(a,c)=1$. The integer $\omega(a,b_0)$ is odd, so
\[
    b=b_0+\bigl(1-\omega(a,b_0)\bigr)c
\]
has the same reduction as $b_0$ and satisfies $\omega(a,b)=1$. Hence
\[
    M=\langle a,b\rangle\oplus\langle a,b\rangle^{\perp_\omega}.
\]
Reducing this decomposition modulo $2$ identifies the second summand with the orthogonal complement of the chosen pair. Induction on $h$ therefore lifts the remaining pairs. The resulting integral symplectic basis $a_i,b_i$ has the prescribed reductions, so every $V(a_i,a_i)$ is even. Together with a basis $x_j$ of $L$, this gives the asserted basis. In particular, all the vectors in the statement are integral.

Set
\[
    S=\sum_{i=1}^h(u_iv_i^T+v_iu_i^T),\qquad V'=V+S.
\]
For each $i,j$, the definition gives $S(a_i,x_j)=-V(a_i,x_j)$ and $S(a_i,b_j)=-V(a_i,b_j)$. For $i\neq j$, it also gives $S(a_i,a_j)=-V(a_i,a_j)$, because $V(a_i,a_j)=V(a_j,a_i)$. Finally, $S(a_i,a_i)=-V(a_i,a_i)$. Hence all the $a$-rows of $V'$ vanish. The $x$-block is unchanged. Since $S$ is symmetric and $V'-V'^T=\omega$, the matrix in the ordered basis $(x,a,b)$ has the form
\[
    V'=
    \begin{pmatrix}
        P&0&B\\
        0&0&0\\
        B^T&-I_h&D
    \end{pmatrix},
    \qquad D=D^T.
\]
It follows that
\[
    tV'-V'^T=
    \begin{pmatrix}
        tP-P^T&0&(t-1)B\\
        0&0&I_h\\
        (t-1)B^T&-tI_h&(t-1)D
    \end{pmatrix}.
\]
Expanding along the $a$-rows and then the $a$-columns gives
\[
    \det(tV'-V'^T)=t^h\det(tP-P^T)=t^{h+d}=t^G.
\]
Here Alexander-triviality gives $\det(tP-P^T)=t^d$ with the usual Seifert-matrix normalization. Finally, $\omega$ is block diagonal with respect to $L\oplus L^{\perp_\omega}$, and the $(a,a)$-block of its inverse is zero. This proves isotropy. Empty summands and matrices are allowed, with the convention that the empty determinant is $1$.
\end{proof}

\begin{corollary}\label{arf-cables}
    Let $n\ge 1$ and let $J$ be a knot with $\Arf(J)=0$. Then $C_{n,1}(J)$ admits a generating sphere in its $n$-trace whose complement has fundamental group $\Z/n\Z$ and which meets the cocore in exactly $2\galg(J)+1$ points. In particular,
    \[
        \shn_n(C_{n,1}(J))\le 2\galg(J)+1.
    \]
    At framing $-n$, the same intersection count and complement group are realized for $\overline{C_{n,1}(J)}$.
\end{corollary}

\begin{proof}
Apply Lemma \ref{arf-certificate}, and then Theorem \ref{lifting} with $m=n$ and $n'=1$. The lifted data for the standard Seifert matrix of $C_{n,1}(J)$ have $\galg(J)$ pairs and monomial determinant. The geometric criterion \ref{geometric-certificate} produces the required sphere. The negative-framing statement follows by reflecting the trace, which preserves the number of cocore intersections and the complement group.
\end{proof}

Corollary \ref{arf-cables} proves Theorem \ref{main-cables}.

The cyclic shake-sliceness assertion in Corollary \ref{arf-cables} is already contained in \cite[Corollary 2.1]{FMNOPR}. The additional conclusion is the specified number of cocore intersections.

\subsection{Realization of all odd shaking numbers}
Fix the handedness of $5_2$ for which a genus-one Seifert matrix is
\[
    V_0=\begin{pmatrix}-1&1\\0&-2\end{pmatrix}.
\]
Its normalized Alexander polynomial and ordinary signature are
\[
    \Delta_{5_2}(t)=2t-3+2t^{-1},\qquad \sigma_{5_2}(-1)=-2.
\]
Indeed, $V_0+V_0^T$ is negative definite. Moreover, $\Delta_{5_2}(-1)=-7\equiv1\pmod8$, so $\Arf(5_2)=0$.

\begin{lemma}\label{five-two}
    For $V_0$, the vectors $u=v=(0,1)^T$ give a framing-$1$ certificate with one pair:
    \[
        tV_0-V_0^T+(t-1)(uv^T+vu^T)
        =
        \begin{pmatrix}
            1-t&t\\
            -1&0
        \end{pmatrix},
        \qquad \det E(t)=t.
    \]
\end{lemma}

\begin{proof}
The displayed identity follows by direct multiplication. The isotropy condition for the single vector $v$ is automatic because $(V_0-V_0^T)^{-1}$ is skew-symmetric.
\end{proof}

\begin{proof}[Proof of Theorem \ref{realization}]
It suffices to prove the result for $n>0$, since reflection identifies the profiles at opposite framings of mirror knots. Put $J_k=\#^k5_2$ and $K_{n,k}=C_{n,1}(J_k)$, with $J_0$ the unknot. The direct sum of $k$ copies of the certificate in Lemma \ref{five-two} has determinant $t^k$. Its $v$-vectors are isotropic because the intersection matrix is block diagonal. Theorem \ref{lifting} and the criterion \ref{geometric-certificate} therefore produce a generating sphere in $X_n(K_{n,k})$ with complement group $\Z/n\Z$ and exactly $2k+1$ cocore intersections. In particular,
\[
    (0,k)\in\cP_n(K_{n,k}).
\]

Write $n=2^a q$ with $a\ge 0$ and $q$ a positive odd integer, and set
\[
    \xi=\exp\left(\frac{2\pi i}{2^{a+1}}\right).
\]
Then $\xi$ has prime-power order $2^{a+1}$ and satisfies $\xi^n=(-1)^q=-1$. Litherland's satellite formula \cite[Theorem 2]{Litherland} gives
\[
    \sigma_{K_{n,k}}(\xi)=\sigma_{J_k}(\xi^n)=-2k,
\]
since the pattern knot $T_{n,1}$ is the unknot. The prime-power root $\xi$ satisfies the admissibility condition of \cite[Proposition 2.17]{CNT}. Applying the topological Murasugi-Tristram inequality \cite[Corollary 3.15]{CNT} directly at $\xi$ gives
\[
    \gfour(K_{n,k})\ge \frac12\lvert\sigma_{K_{n,k}}(\xi)\rvert=k.
\]
Proposition \ref{profile-structure}, applied to $(0,k)$, gives the opposite inequality. Consequently $\gfour(K_{n,k})=k$, and every realized pair $(g,r)$ satisfies $g+r\ge k$.

Conversely, if $0\le r\le k$, applying the tubing move of Proposition \ref{profile-structure} exactly $k-r$ times to $(0,k)$ realizes $(k-r,r)$. Increasing genus realizes every $(g,r)$ with $g+r\ge k$ in this range. If $r\ge k$, finger moves applied to $(0,k)$ realize $(0,r)$, after which genus can again be increased. Therefore
\[
    \cP_n(K_{n,k}) =\{(g,r)\in\Z_{\ge 0}^2 \mid g+r\ge k\}, \qquad \shn_n(K_{n,k})=2k+1.
\]
The initially constructed sphere attains this minimum and has cyclic complement, as required.
\end{proof}

\subsection{Shaking numbers of cables of torus knots}
Feller-Miller-Nagel-Orson-Powell-Ray proved that the cables in the following proposition are topologically shake slice but not smoothly shake slice \cite[Corollary 2.3]{FMNOPR}. The cable bound above determines their topological shaking numbers exactly.

\begin{proposition}\label{top-smooth}
    Let $n,j\ge 1$, and set
    \[
        J=T_{2,8j+1},\qquad K=C_{n,1}(J),
    \]
    where $J$ is the positive torus knot. Then $K$ is $\Z/n\Z$-shake slice and
    \[
        \shn_n(K)=8j+1, \qquad \gfour(K)=4j.
    \]
    However, $K$ is not smoothly $n$-shake slice. The mirror knot has the corresponding properties at framing $-n$.
\end{proposition}

\begin{proof}
The knot $J$ has Seifert genus $4j$, signature of absolute value $8j$ at $-1$, and determinant $8j+1\equiv1\pmod8$. Hence $\Arf(J)=0$, and the signature lower bound gives $\galg(J)=g_3(J)=4j$. Corollary \ref{arf-cables} gives a generating sphere with cyclic complement and $8j+1$ cocore intersections. Choose the same prime-power root $\xi$ with $\xi^n=-1$ as in the proof of Theorem \ref{realization}. Litherland's satellite formula \cite[Theorem 2]{Litherland} gives $\lvert\sigma_K(\xi)\rvert=\lvert\sigma_J(-1)\rvert=8j$. The prime-power admissibility result \cite[Proposition 2.17]{CNT} and the topological Murasugi-Tristram inequality \cite[Corollary 3.15]{CNT} therefore give $\gfour(K)\ge 4j$. Proposition \ref{profile-structure} gives
\[
    2\gfour(K)+1 \le\shn_n(K)\le8j+1,
\]
which proves both equalities.

For completeness, we recall the smooth obstruction used in \cite[Corollary 2.3]{FMNOPR}. Suppose that a smooth sphere $S\subset X_n(K)$ represents a generator of second homology. Its normal disk bundle has Euler number $n$ and boundary $L(n,1)$, oriented as $n$-surgery on the unknot. The inclusion of this disk bundle into $X_n(K)$ is an integral homology equivalence. As in \cite[Lemma 3.1]{FMNOPR}, excision and Poincar\'e-Lefschetz duality show that
\[
    W=X_n(K)\setminus\operatorname{int}\nu S
\]
is a smooth integral homology cobordism between $S^3_n(K)$ and $L(n,1)$.

For $n\ge2$, Gordon's cabling surgery formula \cite[Corollary 7.3]{Gordon}, applied with $p=1$ and $q=n$ in his convention that $q$ is the longitudinal winding number, gives the orientation-preserving identification
\[
    S^3_n(K)\cong L(n,1)\mathbin{\#}Y,
    \qquad Y=S^3_{1/n}(J).
\]
For $n=1$, the same identity follows from $C_{1,1}(J)=J$ and $L(1,1)=S^3$. The integral homology sphere $Y$ has a unique spin$^c$ structure. Restriction of spin$^c$ structures over $W$ gives a bijection between the spin$^c$ structures on its two boundary components. The corresponding Heegaard Floer correction terms agree by their invariance under smooth rational homology cobordism \cite[Proposition 9.9]{OS03}. Additivity under connected sum \cite[Theorem 4.3]{OS03} and summation over the $n$ spin$^c$ structures consequently give
\[
    \sum_{\mathfrak s\in\operatorname{Spin}^c(L(n,1))} d(L(n,1),\mathfrak s)+n\,d(Y)
    =
    \sum_{\mathfrak s\in\operatorname{Spin}^c(L(n,1))} d(L(n,1),\mathfrak s).
\]
Thus $d(Y)=0$.

On the other hand, let $V_0(J)$ denote the nonnegative integer knot Floer invariant in Ni and Wu's surgery formula. Their rational surgery formula \cite[Proposition 1.6]{NiWu} gives $d(Y)=-2V_0(J)$. The standard staircase model for $CFK^\infty(T_{2,2g+1})$ has Maslov-grading-zero generators $x_\ell$, for $0\le\ell\le g$, with filtration coordinates $(\ell,g-\ell)$, each representing the generator of $HF^\infty(S^3)$ in grading zero. This model follows from the filtered-complex calculation in \cite[Section 3]{OS05}. Multiplication by $U$ lowers both filtration coordinates by one. For this staircase, $V_0$ is therefore the minimum, over these generators, of the larger filtration coordinate:
\[
    V_0(T_{2,2g+1})
    =\min_{\ell\in\{0,\ldots,g\}}\max\{\ell,g-\ell\}
    =\left\lceil\frac g2\right\rceil.
\]
Taking $g=4j$ gives $d(Y)=-4j\neq0$, a contradiction. Reflection proves the negative-framing statement.
\end{proof}

\subsection{Alexander polynomials of breadth at most two}
We finish with a restriction that does not depend on the certificate construction. The breadth of a Laurent polynomial is the difference between its largest and smallest exponents.

\begin{proposition}\label{low-degree}
    Let $|n|\ge 2$. If $K$ is $\Z/n\Z$-shake slice and the breadth of $\Delta_K$ is at most two, then $\Delta_K=1$ after the normalization $\Delta_K(t)=\Delta_K(t^{-1})$ and $\Delta_K(1)=1$.
\end{proposition}

\begin{proof}
Set $N=|n|$. Applying the necessary conditions for cyclic shake sliceness in \cite[Propositions 3.3 and 3.5]{FMNOPR} to the sphere-exterior cobordism of \cite[Lemma 3.1]{FMNOPR} gives
\[
    \prod_{\zeta^N=1}|\Delta_K(\zeta)|=1, \qquad \sigma_K(\zeta)=0\quad\text{whenever }\zeta^N=1\text{ and }\zeta\neq 1.
\]
For negative $n$, these statements follow by applying the positive case to the mirror knot. In particular, no $N$th root of unity is a root of $\Delta_K$.

These conditions imply that $\Delta_K(\zeta)>0$ at every nontrivial $N$th root of unity.
Let $V$ be a $2G\times2G$ Seifert matrix for $K$. For $\zeta=e^{i\theta}\neq 1$, write
\[
    H_K(\zeta)=(1-\zeta)V+(1-\bar\zeta)V^T.
\]
The chosen normalization of the Alexander polynomial gives
\[
    \det H_K(e^{i\theta}) =(-1)^G\bigl(4\sin^2(\theta/2)\bigr)^G \Delta_K(e^{i\theta}).
\]
At a nontrivial $N$th root, this Hermitian matrix is nonsingular and has signature zero. It therefore has exactly $G$ negative eigenvalues, and the displayed identity implies $\Delta_K(\zeta)>0$.

The breadth assumption allows us to write
\[
    \Delta_K(t)=a(t+t^{-1})+1-2a,\qquad a\in\Z.
\]
Suppose first that $a\ge 1$. Choose a prime $p$ dividing $N$ and set
\[
    \zeta=\exp\left(\frac{2\pi i\lfloor p/2\rfloor}{p}\right).
\]
Its argument lies in $[2\pi/3,\pi]$, so
\[
    \Delta_K(\zeta) =1-2a+2a\cos\left(\frac{2\pi\lfloor p/2\rfloor}{p}\right) \le1-3a<0.
\]
This contradicts positivity at $N$th roots.

If $a\le-1$, then
\[
    \Delta_K(e^{i\theta})=1+2|a|(1-\cos\theta)>1 \qquad\text{for }e^{i\theta}\neq 1.
\]
Since $N\ge 2$, the product of the absolute values over all $N$th roots is then strictly greater than $1$, another contradiction. Hence $a=0$, as required.
\end{proof}

\section{Signature bounds and profiles without genus reduction}\label{sec:signatures}

For a connected Seifert surface with matrix $V$, write $\sigma_L(z)$ and $\eta_L(z)$ for the signature and nullity of
\[
    H_V(z)=(1-z)V+(1-\bar z)V^T,\qquad z\in S^1\setminus\{1\}.
\]
We first compute these invariants for a shaking. This calculation applies to arbitrary knots, without any shake-sliceness hypothesis.

\begin{proposition}\label{signature-nullity}
    For $n\ge 1$, $r\ge 0$, and $z\in S^1\setminus\{1\}$,
    \begin{align*}
        \sigma_{S_{2r+1,n}(K)}(z)&=\sigma_K(z),\\
        \eta_{S_{2r+1,n}(K)}(z)&=\eta_K(z)+2r\,\mathbf 1_{\{z^n=1\}}.
    \end{align*}
\end{proposition}
\begin{proof}
The Seifert matrix is $V\oplus M_n^{\oplus r}$, as in Section \ref{sec:algebra}. The Hermitian matrix associated to each $M_n$ has the form
\[
    \begin{pmatrix}
        0&B\\
        B^*&0
    \end{pmatrix},
    \qquad
    B=(1-z)I_n+(1-\bar z)Z_n^T.
\]
Its nonzero eigenvalues occur in opposite pairs, and its nullity is $2\dim\ker B$. Since $Z_n^T$ has the $n$-th roots of unity as distinct eigenvalues and
\[
    B=(1-z)(I_n-\bar z Z_n^T),
\]
its kernel has dimension one exactly when $z^n=1$, and dimension zero otherwise. Direct sums give the result.
\end{proof}

\begin{proof}[Proof of Theorem \ref{main-signature}]
By mirroring, it suffices to prove the inequality for $n>0$. Let $F\subset B^4$ be the connected genus-$g$ surface bounded by $S_{2r+1,n}(K)$ supplied by Lemma \ref{shaking}. Then
\[
    \beta_1(F)=2g+2r.
\]
Every nontrivial root of unity of prime-power order satisfies the admissibility condition of \cite[Proposition 2.17]{CNT}. The topological Murasugi-Tristram inequality \cite[Corollary 3.15]{CNT}, applied in the one-colored case to the connected, embedded surface $F$, therefore gives
\[
    |\sigma_{\partial F}(\xi)|+\eta_{\partial F}(\xi)\le\beta_1(F)
\]
for every nontrivial root of unity $\xi$ of prime-power order. The component count in that inequality is the number of connected components of $F$, which is one, rather than its number $2r+1$ of boundary components. If $\xi^n=1$, Proposition \ref{signature-nullity} therefore yields
\[
    |\sigma_K(\xi)|+\eta_K(\xi)+2r\le2g+2r.
\]
This proves the claimed lower bound. At the prime-power roots under consideration, the knot nullity is in fact zero. Indeed, if $\xi$ has order $p^a$ and $\Delta_K(\xi)=0$, then the cyclotomic polynomial $\Phi_{p^a}(t)$ divides a polynomial representative of $\Delta_K(t)$. Evaluating at $t=1$ would force $p=\Phi_{p^a}(1)$ to divide $\Delta_K(1)=1$, a contradiction.

For the vertical profiles, use the matrix $V_0$ from Lemma \ref{five-two}. If $z=e^{i\theta}$ and $a=2-2\cos\theta$, then
\[
    H_{V_0}(z)=
    \begin{pmatrix}
        -a&1-z\\
        1-\bar z&-2a
    \end{pmatrix},\qquad
    \det H_{V_0}(z)=a(2a-1),\qquad
    \operatorname{tr}H_{V_0}(z)=-3a.
\]
Choose any prime $p$ dividing $|n|$ and put
\[
    \xi=\exp\left(\frac{2\pi i\lfloor p/2\rfloor}{p}\right).
\]
Its argument lies in $[2\pi/3,\pi]$, so $a\ge 3$ and $H_{V_0}(\xi)$ is negative definite. Hence
\[
\sigma_{\#^j5_2}(\xi)=-2j.
\]
The bound already proved gives $g\ge j$ for every pair in the profile. A genus-$j$ Seifert surface capped by the handle core realizes $(j,0)$. Proposition \ref{profile-structure}(a) then realizes every pair with $g\ge j$. This also treats $j=0$.

When $n$ is even, take $\xi=-1$ and use
\[
    |\sigma_{T_{2,2j+1}}(-1)|=2j=g_3(T_{2,2j+1})\cdot2.
\]
The same lower-bound and realization argument gives the stated torus-knot profiles.
\end{proof}

\begin{remark}
The algebraic genus of $5_2$ is one: it is at most its Seifert genus and at least half the absolute ordinary signature. Since $\Arf(5_2)=0$, the framing-one construction realizes $(0,1)$. Theorem \ref{main-signature} shows that $(0,r)$ is absent for every $r$ and every $|n|\ge 2$. Thus the literal framing-one genus-intersection trade-off fails for all these framings.
\end{remark}

\section{Scope of the certificates and further questions}\label{sec:scope}

\subsection{An Alexander-polynomial restriction}
The determinant criterion imposes the following restriction on the normalized Alexander polynomial, which is stronger than nonvanishing at the $n$-th roots of unity.

\begin{proposition}\label{certificate-restriction}
    Suppose the criterion of Corollary \ref{geometric-certificate} holds at framing $n\ge 1$. Then
    \[
        \Delta_K(t)\equiv1\pmod{t^n-1}
    \]
    in $\Z[t,t^{-1}]$.
\end{proposition}
\begin{proof}
Let $V$ be the Seifert matrix in the certificate, of size $2g$, and let $E(t)$ be the matrix in \eqref{eq:certificate-matrix}. The normalization established after Theorem \ref{det-formula} gives $\det E(t)=t^g$. The correction term in $E(t)$ vanishes in $\Z[t,t^{-1}]/(t^n-1)$, so
\[
    t^g\Delta_K(t)=\det(tV-V^T)\equiv\det E(t)=t^g\pmod{t^n-1}.
\]
Since $t$ is a unit in this quotient, multiplication by $t^{-g}$ proves the assertion.
\end{proof}

Equivalently, the image of $\Delta_K$ in the integral group ring $\Z[\Z/n\Z]$ is the identity element. In particular, it is a trivial unit and its class in the Whitehead group is zero.
We prove this restriction under the certificate hypothesis of Corollary \ref{geometric-certificate}.
The cables in Theorem \ref{main-cables} satisfy the congruence because
\[
    \Delta_{C_{n,1}(J)}(t)=\Delta_J(t^n)\equiv1\pmod{t^n-1}.
\]

\begin{question}
    For a fixed $n\ge2$, which $\Z/n\Z$-shake slice knots admit a certificate of the form in Corollary \ref{geometric-certificate}? Can a broader Seifert-form construction give bounds for the shaking number when the congruence in Proposition \ref{certificate-restriction} fails?
\end{question}

\subsection{The zero-framing case}
Our construction does not apply at framing zero. The known genus formula at this framing concerns generating surfaces with infinite cyclic complement.
Feller-Park-Powell \cite[Theorem 1.5]{FPP} proved that the minimum genus of a generating surface $\Sigma\subset X_0(K)$ satisfying
\[
    \pi_1(X_0(K)\setminus\Sigma)\cong\Z
\]
equals the $\Z$-genus of $K$, namely the minimum genus of a connected, oriented, locally flat surface in $B^4$ bounded by $K$ with infinite cyclic complement. The meridian generates the complement group under the displayed condition, by the homology calculation in Section \ref{sec:profiles}. The $\Z$-genus equals $\galg(K)$ by \cite[Corollary 1.5]{FL19}. Consequently the equality with $\galg$ applies to the cyclic-complement version of shake genus, not to ordinary shake genus.

The distinction is visible even for slice knots. If $K=J\#\bar J$ with $J$ a trefoil, then $K$ is slice and has a genus-zero generating surface in $X_0(K)$, whereas $\Delta_K\neq 1$ and hence $\galg(K)>0$. No unrestricted inequality $\gsh^0(K)\ge \galg(K)$ can therefore hold.

\begin{question}
    Does $\shn_0(K)\le3$ imply that $K$ is topologically slice?
\end{question}

Since $\shn_0(K)=1$ precisely when $K$ is topologically slice and every finite shaking number is odd, this asks whether $\shn_0(K)=3$ can occur. It is the first possible nontrivial case of the question whether every topologically $0$-shake slice knot is slice, recorded in \cite[Problem 1.57(a)]{BKR26}.

\subsection{Other profiles}
Theorems \ref{realization} and \ref{main-signature} give examples in which the minimum realized genus decreases at each step until reaching zero, or remains constant. Proposition \ref{profile-structure} shows that the sequence $r\mapsto\gamma_n(K;r)$ is nonincreasing and drops by at most one at each step. It remains to determine which sequences satisfying these restrictions can occur.

\begin{question}
    For a fixed nonzero framing, which nonincreasing sequences of nonnegative integers, with successive drops at most one, occur as $r\mapsto\gamma_n(K;r)$? In particular, can one prescribe several intervals on which the sequence is constant, separated by nonzero drops?
\end{question}

\begin{question}
    For which knots $J$ with vanishing Arf invariant is the upper bound
    \[
        \shn_n(C_{n,1}(J))\le2\galg(J)+1
    \]
    sharp for every $n\ge 1$? Can the shaking numbers of these cables vary with $n$ for a fixed companion $J$?
\end{question}

\paragraph{Tool Use} AI tools were used during the final proofreading and revision of this manuscript to assist with checking calculations, improving the English exposition, and reviewing the literature on the current status of the open questions discussed in Section \ref{sec:scope}.

\bibliographystyle{amsalpha}
\bibliography{references}
\end{document}